\documentclass[twoside]{aiml22}

\usepackage{aiml22macro}

\usepackage{microtype}
\usepackage[british]{babel}
\usepackage{graphicx}
\usepackage{amsmath}
\usepackage{amssymb}
\usepackage{mathtools}
\usepackage{cleveref}
\usepackage[shortlabels]{enumitem}

\usepackage{tikz}
\usetikzlibrary{calc}
\usetikzlibrary{arrows.meta}
\usetikzlibrary{decorations.markings}

\newcounter{saveenumerate}
\makeatletter
\newcommand{\enumeratext}[1]{%
\setcounter{saveenumerate}{\value{enum\romannumeral\the\@enumdepth}}
\end{enumerate}
#1
\begin{enumerate}[I]
\setcounter{enum\romannumeral\the\@enumdepth}{\value{saveenumerate}}%
}
\makeatother

\newcommand{\reals}{\mbox{\(\mathbb R\)}}

\def\centerarc[#1](#2)(#3:#4:#5)
    { \draw[#1] ($(#2)+({#5*cos(#3)},{#5*sin(#3)})$) arc (#3:#4:#5); }

\newcommand{\lleft}{\var{left}} 
\newcommand{\rright}{\var{right}} 

\newcommand{\modal}{\mathbf}
\newcommand{\F}{\modal F}
\renewcommand{\P}{\modal P}
\newcommand{\G}{\modal G}
\renewcommand{\H}{\modal H}
\newcommand{\oor}{\vee}
\newcommand{\aand}{\wedge}
\newcommand{\limplies}{\rightarrow}
\newcommand{\var}{\mathit}
\newcommand{\eloise}{\var{eloise}}

\newcommand{\pos}{\var{\beta}}

\newcommand{\white}{\var{white}}
\renewcommand{\And}{\bigwedge}
\newcommand{\Or}{\bigvee}

\newcommand{\col}{\var{col}}

\newcommand{\ind}{\var{index}}
\newcommand{\rank}{\var{depth}}

\newcommand{\up}{\var{up}}
\newcommand{\down}{\var{down}}
\newcommand{\win}{\var{win}}
\newcommand{\counter}{\var{counter}}

\newcommand\set[1]{{ \{#1\} }}

\renewcommand \c[1]{{\mathcal #1}}

\newcommand\UM{{\blacksquare}}

\def\lastname{Hirsch and McLean}

\begin{document}

\begin{frontmatter}
  \title{EXPTIME-hardness of higher-dimensional Minkowski spacetime}
  \author{Robin Hirsch}\footnote{r.hirsch@ucl.ac.uk}
  \address{Department of Computer Science, University College London \\ Gower Street \\ London WC1E 6BT \\ United Kingdom}
  \author{Brett McLean}\footnote{brett.mclean@ugent.be \\The second author was supported by the Research Foundation -- Flanders (FWO) under the SNSF--FWO Lead Agency Grant 200021L 196176 (SNSF)/G0E2121N (FWO).}
  \address{Department of Mathematics: Analysis, Logic and Discrete Mathematics, Ghent University\\ Building S8, Krijgslaan 281 \\ 9000 Ghent \\ Belgium}

  \begin{abstract}
  We prove the EXPTIME-hardness of the validity problem for the basic temporal logic on Minkowski spacetime with more than one space dimension. We prove this result for both the lightspeed-or-slower and the slower-than-light accessibility relations (and for both the irreflexive and the reflexive versions of these relations). As an auxiliary result, we prove the EXPTIME-hardness of validity on any frame for which there exists an embedding of the infinite complete binary tree satisfying certain conditions. The proof is by a reduction from the two-player corridor-tiling game.
  \end{abstract}

 \end{frontmatter}

\section{Introduction}
     
Temporal logic has traditionally treated time as absolute, independent of other aspects of state, such as position. However, special and general relativity tell us that no absolute view of time reflects physical reality, and we must be content with understanding space and time jointly in the form of a unified spacetime. Thus the study of temporal logics of spacetime can be seen as a natural and foundational topic in temporal logic.

The frames we study in this paper are: Minkowski spacetime with \emph{lightspeed-or-slower} accessibility and Minkowski spacetime with \emph{slower-than-lightspeed} accessibility. Minkowski spacetime is the flat spacetime of special relativity, where light travels in straight lines. 
Let $m$ be the number of space dimensions. Then Minkowski spacetime with (irreflexive) lightspeed-or-slower accessibility can be realised as the Kripke frame $(\mathbb R^{m+1}, <)$ where \[(x_1, \dots, x_m, t) < (y_1, \dots, y_m, t') \iff t < t' \text{ and }\sum_i (x_i -y_i)^2 \leq (t'-t)^2.\]
(Here the order relation on the right-hand side is the usual ordering on the reals.) Minkowski spacetime with (irreflexive) slower-than-lightspeed accessibility can be realised as the frame $(\mathbb R^{m+1}, \prec)$ where \[(x_1, \dots, x_m, t) \prec (y_1, \dots, y_m, t') \iff t < t' \text{ and }\sum_i (x_i -y_i)^2 < (t'-t)^2.\]
Depictions of these frames for $m = 2$ can be found in \Cref{fig:space}.
\begin{figure}[h]
\begin{center}
\begin{tikzpicture}[scale=1]
\draw[dotted, white, fill=blue!10] (0,0) -- ++(-2,2) -- ++(4,0) -- ++(-2,-2);
\draw[dotted, white, fill=brown!30] (0,0) -- ++(-2,-2) -- ++(4,0) -- ++(-2,2);
\draw [->] (-2,-.5) -- (2, -.5);
\draw[->] (-1, -2) -- (-1, 2);
\draw[->] (-2, -1) -- (2, 1);
\draw(-2, -2) -- (2,2);
\draw(-2, 2) -- (2, -2);
\node [right] at (2,-.5) {$x_1$};
\node [right] at (2,1) {$x_2$};
\node [above] at (-1,2) {$t$};
\draw[dashed] (0,1) ellipse (.95cm and .3cm);
\draw[dashed] (-1.90,2) arc(180:360:1.9cm and .6cm);
\draw[dashed] (0,-1) ellipse (.95cm and .3cm);
\draw[dashed] (-1.90,-2) arc(180:0:1.9cm and .6cm);

\draw [fill] (0,0) circle [radius=2pt];
\end{tikzpicture}
\hspace{1cm}
\begin{tikzpicture}[scale=1]
\draw[dotted, white, fill=blue!10] (0,0) -- ++(-2,2) -- ++(4,0) -- ++(-2,-2);
\draw[dotted, white, fill=brown!30] (0,0) -- ++(-2,-2) -- ++(4,0) -- ++(-2,2);
\draw [->] (-2,-.5) -- (2, -.5);
\draw[->] (-1, -2) -- (-1, 2);
\draw[->] (-2, -1) -- (2, 1);
\draw[dotted](-2, -2) -- (2,2);
\draw[dotted](-2, 2) -- (2, -2);
\node [right] at (2,-.5) {$x_1$};
\node [right] at (2,1) {$x_2$};
\node [above] at (-1,2) {$t$};
\draw[dashed] (0,1) ellipse (.95cm and .3cm);
\draw[dashed] (-1.90,2) arc(180:360:1.9cm and .6cm);
\draw[dashed] (0,-1) ellipse (.95cm and .3cm);
\draw[dashed] (-1.90,-2) arc(180:0:1.9cm and .6cm);

\draw [fill] (0,0) circle [radius=2pt];
\end{tikzpicture}
\end{center}
\caption{\label{fig:space}Minkowski spacetime frames $(\mathbb R^{3}, <)$ and $(\mathbb R^{3}, \prec)$}
\end{figure}
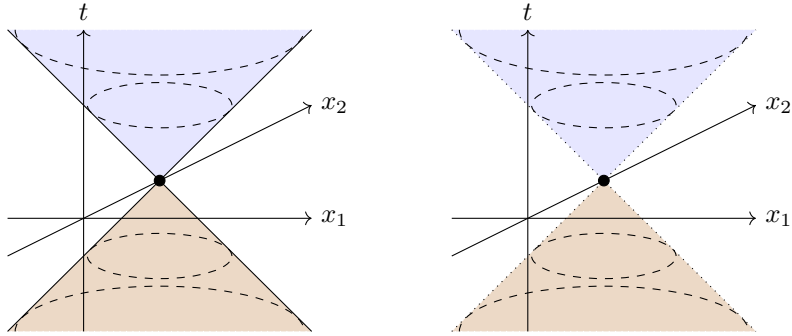

One may also consider reflexive versions $\leq$ and $\preceq$ of these accessibility relations, given by the reflexive closure of the respective irreflexive relations. This distinction will not play an important role in this paper.

The language we use is the most classical temporal language, the language of Prior's tense logic \cite{prior1957time}. Formulas are built from a countable supply $p, q, \dots$ of propositional variables using the standard propositional connectives $\neg$ and $\vee$ and the unary temporal operators $\F$ (`at some time in the future') and $\P$ (`at some time in the past'). We call the formulas of this language $(\F, \P)$-formulas. Usual classical propositional abbreviations apply, and $\G \coloneqq\neg \F \neg$ and $\H \coloneqq \neg \P \neg$.
A \emph{modal} formula is built using propositional connectives and $\F$ (but no $\P$).

It is known that the  modal validities of reflexive Minkowski spacetime with $m$ spatial dimensions are axiomatised by $\modal{S4.2}$ (reflexive, transitive, confluent)  \cite{Gol80}, regardless of choice of $m\geq 1$, and regardless of whether speed-of-light is allowed; hence the logic is PSPACE-complete \cite{Shap05}.   For irreflexive spacetime, with slower-than-light accessibility the validities again do not depend on the dimension, and are given by a logic called $\modal {OI.2}$ (transitive, confluent, serial, two-dense: $\F p \wedge \F q \limplies \F(\F p \wedge \F q)$) \cite{DBLP:conf/aiml/ShapirovskyS02}, which is also PSPACE-complete \cite{Shap05}. With irreflexive lightspeed-or-slower accessibility, the logics depend on the dimension, and there is a formula satisfiable with two  spatial dimensions yet not with only one spatial dimension \cite{Gol80}.  By a simple reduction from the reflexive cases it is clear that  the modal logic of irreflexive lightspeed-or-slower accessibility is PSPACE-hard in all cases $m\geq 1$, though we could not find a proof in the literature that in all cases the logic is in PSPACE.

For temporal validities, (i.e.\ valid $(\F, \P)$-formulas) over 
two-dimensional Minkowski spacetime with lightspeed-or-slower accessibility, that is, either $(\mathbb R^{2}, <)$ or $(\mathbb R^{2}, \leq)$,   validity is PSPACE-complete \cite{HR18}, and  the same is true for slower-than-light accessibility \cite{HM18}.\footnote{A contrasting result is that with \emph{exactly lightspeed} accessibility, validity is undecidable \cite{shapirovsky2010simulation}.} However, little is known about the logics of the higher-dimensional frames, except that they differ from those in two dimensions \cite{HR18}. In the problem (5) at the conclusion of \cite{HR18}, Hirsch and Reynolds ask for the decidability of the $(\mathbb R^{3}, <)$ validities and conjecture undecidability.

In this paper, we make a small amount of progress towards identifying the complexity of the validity problem in higher dimensions, by proving an EXPTIME lower bound (with any of the four accessibility relations: lightspeed-or-slower or slower-than-light and reflexive or irreflexive). We do this by reducing the two-person corridor-tiling game to satisfiability in our higher-dimensional frames. We describe this tiling game, which is known to be EXPTIME-complete \cite{CHLEBUS1986374}, in the following section. In \Cref{sec:reduction} we carry out the reduction, giving a general EXPTIME-hardness result (Theorem~\ref{thm:main}), and in \Cref{sec:minkowski} we apply this theorem to Minkowski spacetime frames (Theorem~\ref{thm:spacetime}). \Cref{sec:problems} lists open problems.

\section{The two-person corridor-tiling game}

In order that this paper be self contained, in this section we present the two-person corridor-tiling game. Our version of the game is that found in \cite[\S 6.8]{blackburn_rijke_venema_2001}.

A Wang tile type, $T$, is a square with its sides labelled by colours $\lleft(T)$, $\rright(T)$, $\up(T)$ and $\down(T)$ \cite{6773658}. From this presentation it is clear that tiles are `fixed in orientation', that is, what is obtained by rotating a tile by $90^\circ$ is considered distinct from the original tile type. \Cref{fig:Wang} shows the conventional way of depicting a Wang tile.

\begin{figure}[h]
\begin{center}
\begin{tikzpicture}
\path[ fill=gray!12] (0,0) -- ++(-1,1) -- ++(2,0) -- ++(-1,-1);
\path[ fill=gray!60] (0,0) -- ++(-1,-1) -- ++(2,0) -- + (-1,1);
\path[ fill=blue!20] (0,0) -- ++(-1,1) -- ++(0,-2) ;
\path[ fill=brown!30] (0,0) -- ++(1,-1) -- ++(0,2) ;
\end{tikzpicture}
\end{center}
\caption{\label{fig:Wang}Depiction of a Wang tile}
\end{figure}
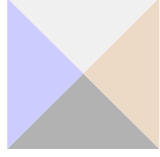

\noindent\emph{Game instances.} An instance of the two-player corridor-tiling game is a pair  $((T_0, \dots, T_{s+1}),\allowbreak (I_1, \dots, I_n))$, where $(T_0, \dots, T_{s+1})$ is a sequence of tile types and $(I_1, \dots, I_n)$ is a sequence of tiles drawn from $\{T_0, \dots, T_{s+1}\}$.

\smallskip

\noindent\emph{The players and board.} The game is played between two players---Eloise and Abelard---on an $n \times \omega$ grid---the `corridor'. There are walls to the left of the first column and to the right of the $n$th column, which we consider to be columns $0$ and $n+1$ respectively, and throughout the game both these columns are filled entirely with instances of tile type $T_0$. We assume this gives both walls the same colour, $\white$ say. At the start of the game the first row of the
corridor is filled with the tiles $(I_1, \ldots, I_n)$ in that order.  (This is the only purpose of the data $(I_1, \dots, I_n)$, beyond determining the value of $n$.\footnote{Note though that the presence of $(I_1, \dots, I_n)$ in instance specification ensures the corridor width is polynomial---not exponential---in the size of the input, matching the stipulation in Chlebus' original presentation that $n$ be encoded in unary.}) This initial position is depicted in \Cref{fig:corridor}.
\begin{figure}[h]
\begin{center}
\begin{tikzpicture}
\draw (0,0)--(5,0);
\draw[->] (0,0)--(0,3);
\draw[->] (5,0)--(5,3);
\draw(0,1)--(5,1);
\draw (1,0)--(1,1);
\draw(2,0)--(2,1);
\draw(4,0)--(4,1);
\draw[dashed](1,1)--(1,2)--(0,2);

\node at (-.5,.5) {$T_0$};
\node at (-.5,1.5) {$T_0$};
\node at (-.5,2.5) {$T_0$};
\node at (5.5,.5) {$T_0$};
\node at (5.5,1.5) {$T_0$};
\node at (5.5,2.5) {$T_0$};

\node at (.5,.5) {$I_1$};
\node at (1.5,.5) {$I_2$};
\node at (4.5,.5) {$I_n$};
\node at (3,.5){$\cdots$};
\node (m) at (2,2.35){Eloise plays next tile};
\draw[->](m)--(.5,1.5);

\end{tikzpicture}
\end{center}
\caption{\label{fig:corridor}Game instance $((T_0, \dots, T_{s+1}),\allowbreak (I_1, \dots, I_n))$ at the start of play}
\end{figure}
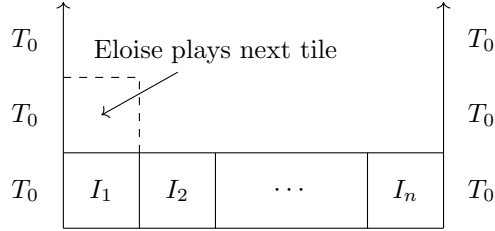

\smallskip

\noindent\emph{The play.} During play, the two players alternate, with Eloise playing first. To make a legal move, a player must chose a tile type and place one instance of that tile type on the board. The placement of the new tile is fixed: it is placed in the leftmost available position in the first incomplete row. It is required that the edge-colours of the placed tile match up with those of adjacent tiles already on the board.

\smallskip

\noindent\emph{Winning conditions.} If at any point during the play, tile type $T_{s+1}$ is placed in the first column then the game ends and Eloise wins. 
 In all other cases,  Abelard wins; there are two ways this can happen.   Firstly, if $T_{s+1}$ has not yet been placed in the first column, and the current player (Eloise or Abelard) is unable to make a move then the game ends and Abelard wins.  
  The other possible way Abelard can win is if the  play goes on infinitely long without $T_{s+1}$ being placed in the first column, and where each player always makes a legal move.

\smallskip

Based on these rules, in order to have all relevant information about the state of play it is sufficient to know the \emph{position} $(p, \bar T, i)$ consisting of:
\begin{enumerate}
\item
the player $p\in\set{\exists,\forall}$ whose turn it is,
\item
 a sequence $\bar T$ of $n$ tiles, indicating the last tile placed in columns $1$ to $n$, 
 \item
  a column index $i$, where $1\leq i\leq n$, indicating the column where the next tile must be placed. 
	  \end{enumerate}


\smallskip

\noindent\emph{The game tree and strategy.} Instances of the corridor-tiling game may analysed via their game tree. Let $\c I = ((T_0,\ldots, T_{s+1}),(I_1, \ldots, I_n))$ be an instance.  Let $p_0=(\exists, (I_1,\ldots, I_n), 1)$ be the initial position in  the game tree $\c T(\c I)$.  Each node of $\c T(\c I)$ is labelled by a position, and the root of $\c T(\c I)$ is labelled by the initial position $p_0$.  In any position there are at most $s+2$ possible moves (as there are $s+2$ tile types).   Each possible move determines the position of a child node. Each branch 
 is either a win for Eloise or Abelard, not both.  

A \emph{strategy} for Eloise is a subtree of the game tree including the root,  all children of  included nodes where it is Abelard's turn, and at least one child of included nodes where it is Eloise's turn.  Such a strategy is a winning strategy if   all branches are wins for Eloise.  Of course, whether such a winning strategy exists depends on the instance. The decision problem associated to the two-person corridor-tiling game is the set of instances for which Eloise has a winning strategy.

The two-person corridor-tiling game was first presented in \cite{CHLEBUS1986374}, where it was called the \emph{rectangle tiling game} and was proven to be EXPTIME-complete.


\section{Reduction from corridor-tiling game}\label{sec:reduction}
In this section we take an arbitrary instance $\c I = ((T_0,\ldots, T_{s+1}),(I_1, \ldots, I_n))$ of the two-person corridor-tiling game and compute an associated temporal formula $\phi_\c I$. Then we show that for any Kripke frame into which there is a certain kind of embedding of a complete binary tree with no leaves, $\c I$ is a yes-instance if and only if $\phi_\c I$ is satisfiable in the frame. 

The rough idea is that we can use the image of the embedded tree as a scaffolding for drawing the game tree $\c T(\c I)$ in the frame, simply by labelling (images of) nodes using propositional variables. (We do however re-encode $\c T(\c I)$, with its finite branching factor, as a binary-branching tree, in order that the embedding be easier to produce.) Then $\phi_\c I$ expresses that a given point is the root of a game tree for which Eloise has a winning strategy.

 We first introduce some propositional formulas and definable modalities, to help express the formula $\phi_\c I$.

In Minkowski spacetime we may express the universal modality $\UM$ by
\[
\UM\varphi:= \G\H\varphi.\]
More generally, let $(V, <)$ be any frame that is both \emph{transitive} and  \emph{confluent} ($x < y, z$ implies there is $w$ with $y, z < w$). 
 Then $\UM$ is `for all within the current weakly connected component' (of $(V, <)$, viewed as a directed graph), which is strong enough for our purposes.\footnote{Technically this is not true if $<$ is empty on a (necessarily singleton) connected component. If this occurs in the frame then one can add a conjunct $\F\top$ to the $\phi_\mathcal I$ we describe to prevent the formula being satisfied at such a point.} Note that we will refer to elements of $V$ as points.

We will use the following propositional variables. 
\begin{enumerate}
\item $f$ to denote a `forbidden set' $F$.
\item $\ind_i$, for $1\leq i\leq n$, to indicate a node of the tree at which the next tile is to be played in column $i$. (If no $\ind_i$ holds at a point in $V$ then the point is not a node.)
\end{enumerate}
Later, in the definition~\eqref{eq:loz}, we use $F$ to help define a modality $\lozenge$ (see also Figure~\ref{fig:loz}).

The remaining variables give information only at nodes and have no significance elsewhere (where they may be assigned arbitrarily).
\begin{enumerate}\setcounter{enumi}{3}
\item
$\col_i(T)$, for all $0 \leq i \leq n+1$ and all $T \in \{T_0, \dots, T_{s+1}\}$, used to indicate that the latest tile placed in column $i$ was of type $T$.
\item
$\eloise$, true if it is Eloise's turn to play, false if Abelard's.
\item\label{last}
$\win$ to indicate that if play continues from the current position, Eloise has a winning strategy.

\end{enumerate}

For the encoding of $\c T(\c I)$ as a labelled binary tree $\c B$, the maximum number of children of a node of $\c T(\c I)$ is $s+2$; let  $b = \lceil \log_2 (s+2)\rceil$.   By duplicating branches, we may assume that the branching factor of $\c T(\c I)$ is always either $2^b$ or zero.  We also assume $b\geq 3$.

Let $\Delta$ be a complete binary tree of depth $b$, whose nodes are strings over $\set{0, 1}$ of length at most $b$, and with edges $(x, x0)$ and $(x, x1)$ when $x$ is a string of length less than $b$. The root of $\Delta$ is the empty string $\varepsilon$.  Each node of $\Delta$ has a depth equal to the length of the bit-string, so the root has depth $0$ and the leaves have depth $b$.   

 For the \emph{tree structure} of $\c B$, first take a forest consisting of copies $\Delta_p$ of $\Delta$ indexed by nodes of $\c T(\c I)$, so a typical node will be $(p, x)$ where $p\in\c T(\c I)$ and $x\in\Delta$. Now for every non-leaf node $p$ of $\c T(\c I)$ choose a bijection $\theta_p$ between the leaves of $\Delta_p$ and those nodes $q$ of $\c T(\c I)$ where there is a move from $p$ to $q$ in $\c T(\c I)$ (each set has cardinality $2^b$). Now, for each non-leaf node $p$ of $\c T(\c I)$ identify each leaf $(p, x)$ of $\Delta_p$ with $(\theta_p(p,x), \varepsilon)$. Finally, remove any leaf nodes (precisely the nodes of depth $b$ that were not identified). 
   This defines the binary tree structure of $\c B$.      After this identification,  the nodes are $(p, x)$ where $p$ is a node of $\c T(\c I)$, and the node $x$ of $\Delta$ has depth at most $b-1$.   The children of $(p, x)$  are $(p, x0), (p, x1)$ when the depth of $x$ is less than $b-1$.  When the depth of $x$ is $b-1$ the children of $(p, x)$ are $(q, \varepsilon)$ where $q$ can be one of two successor positions of $p$ in $\c T(\c I)$.  
  For the \emph{labelling} of $\c B$, each node $(p, x)$ is labelled by the node $p$ of $\c T(\c I)$ and the depth of $x$.   
  
  We thus also have propositions
  \begin{enumerate}\setcounter{enumi}{6}
  \item $\rank_j$, for $0\leq j\leq b-1$, to indicate the depth of $x$ in $\Delta$.
  \end{enumerate} 
  
Let
\begin{align*}
\beta (i, j)&:=\ind_i\wedge\rank_j,
\shortintertext{and}
\beta&:=\bigvee_{1\leq i\leq n}\ind_i.
\end{align*}
Thus $\beta$ is the proposition `is a node'.

Define a function $^+:[1,n]\times[0,b-1] \to [1,n]\times[0,b-1]$ that updates the column-index and depth by either incrementing the depth if less than $b-1$ leaving  column-index fixed,  else updating the   column-index and resetting depth to $0$.
That is, for $(i, j)\in[1,n]\times[0,b-1]$, let 
\[(i, j)^+=\left\{\begin{array}{ll} (i, j+1)&j\leq b-2\\
(i+1,0)&j=b-1,\; i<n\\
(1,0)&j=b-1,\; i=n\,.
\end{array}
\right.
\]
 If $(i, j)$ is  the index and depth in an encoded version of the game where  
  $b$ levels encode one move, then  $(i, j)^+$ is  the index and depth of children in this encoded game.  Since $b\geq 3$ we know that $(i, j)^{++}\neq(i,j)$ (this allows us to distinguish the children from the parent of any node).
 
\smallskip

We will now define a pair of modalities  $\lozenge, \Box$ that will express `possibly for a child' and `necessarily for children' respectively, and thus will do most of the heavy lifting in the formula $\phi_\c I$. It may be helpful to know
that for the Minkowski frames, we will be embedding the binary tree into the
(hyper)plane $t = 0$. 
To give some intuition, 
  Figure~\ref{fig:loz} illustrates a point where $\lozenge\varphi$ holds. 
  
  We define $\lozenge$ and $\Box$ as follows.
\begin{align}
\label{eq:loz}
\lozenge \varphi:&= \bigvee_{1\leq i\leq n,\; 0\leq j\leq b-1} \beta(i, j) \wedge \F(\P(\beta((i, j)^+)\wedge\varphi)\wedge\H\neg f)\\
\Box\varphi:&= \bigwedge_{1\leq i\leq n,\; 0\leq j\leq b-1} \beta(i, j) \limplies\G(\H\neg f \limplies \H(\beta((i, j)^+)\limplies\varphi)) \nonumber
\end{align}
Note that $\Box\varphi$ is indeed $\modal K_t$-equivalent to $\neg\lozenge \neg\varphi$, and that ${\modal K_t}\vdash(\Box\varphi \aand \Box\psi )\limplies \Box (\varphi \aand \psi)$.\footnote{Recall that $\modal K_t$ is the temporal analogue of modal $\modal K$, that is, the set of temporal formulas valid on all frames.}

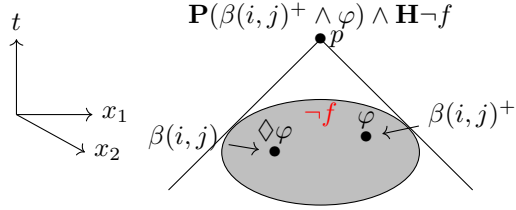
\begin{figure}
\begin{center}
\begin{tikzpicture}[scale=1]
\draw[dotted, white] (0,0) -- ++(-2,-2) -- ++(4,0) -- ++(-2,2);

\draw[->] (-4,-1) to (-4,0);
\draw[->](-4,-1) to (-3,-1);
\draw[->](-4,-1) to (-3.1,-1.5);

\node [above] at (-4,0){$t$};

\node [right] at (-3,-1){$x_1$};
\node[right] at (-3.1,-1.5){$x_2$};
\draw(-2, -2) -- (0,0);
\draw(0,0) -- (2, -2);


\begin{scope}
\path [inner color=white,   opacity=0, outer color=white] (0cm,-1.5cm) {ellipse(2.2cm and 1.1cm)};
\end{scope}
\draw[fill=gray!50] (0cm,-1.5cm) {ellipse(1.3cm and .7cm)};

\node(x) at (0,0){$\bullet$};
\node[right] at (x) {$p$};

\node(m) at (-.6,-1.5){$\bullet$};
\node[above] at (m){$\lozenge\varphi$};

\node(n) at (.6,-1.3){$\bullet$};
\node[below] at (n){};
\node[above] at (n){$\varphi$};


\node(r) at (2, -1) {$\beta(i, j)^+$};

\draw[->] (r) to (n);

\node[red] at (0,-1){$\neg f$};

\node[above] at (x) {$\P(\beta(i, j)^+\wedge\varphi)\wedge \H \neg f$};

\node(l) at (-1.8, -1.3){$\beta(i,j)$};

\draw[->] (l)  to (m);

\end{tikzpicture}
\caption{\label{fig:loz} The formula $\lozenge\varphi$ holds at a point if some $\beta(i, j)$ holds there, and a point $p$ in the future sees a point in its past where $\varphi\wedge\beta(i, j)^+$ holds, but no point satisfying $f$ is in the past of $p$. In particular, $f$ must not hold at any point in the shaded disc where the past of $p$ intersects the spatial plane. }
\end{center}
\end{figure}

Now is a good time to state our main theorem.

\begin{theorem}\label{thm:main}
Let $(V, <)$ be any  transitive and confluent 
 Kripke frame, and let $\c B$ be the infinite complete binary tree. (We view the edges, $E(\c B)$, of $\c B$ as directed from parent to child.) Suppose there is a map ${^\prime}:\c B \to V$
 and a subset $F\subseteq V$ such that the following condition holds for all distinct $x, y\in \c B$.
\begin{equation}\label{cond}
\begin{gathered}
 (x, y) \in E(\c B)\vee (y, x)\in E(\c B)\\ \iff\\  \exists z \in V (z>x'\wedge z>y'\wedge \forall w(w<z \implies w\notin F))
 \end{gathered}
\end{equation}
Then the satisfiability of $(\F, \P)$-formulas over $(V, <)$  is EXPTIME-hard.
\end{theorem}
Note that the map $'$ of Theorem~\ref{thm:main} is necessarily injective, since if $x_1' = x_2'$ then condition~\eqref{cond} implies $x_1$ and $x_2$ have (excluding each other) the same neighbours in $\mathcal{B}$.

As mentioned previously, given an instance $ \c I = ((T_0, \dots, T_{s+1}),\allowbreak (I_1, \dots, I_n))$ of the two-person corridor-tiling game, 
 we construct a formula $\phi_{\c I}$ that is satisfiable in $(V, <)$ if and only if Eloise has a winning strategy for $\mathcal T(\c I)$.  
The formula $\phi_{\c I}$ will state that the initial position is a winning position for Eloise.

The formula $\phi_{\c I}$ is the conjunction of all the conditions \ref{form:first}--\ref{form:last} we are about to lay out.  We follow the structure of the presentation of the reduction to PDL satisfiability given on pages 398--401 of \cite{blackburn_rijke_venema_2001}. That reduction originates from \cite{CHLEBUS1986374}. Most of the conditions have the $\UM$ modality in outermost position, so that what is inside this modality holds throughout the play. We omit the phrase `Throughout the play\dots' from the textual descriptions of these conditions.
\medskip

\begin{enumerate}[I]

\item\label{form:first} Position and depth at start of play:
\[\eloise \aand \beta(1,0)\aand \col_0(T_0) \aand \col_1({I_1}) \aand \dots \aand \col_n({I_n}) \aand \col_{n+1}(T_0).\]

\item {Nodes have a unique index:}
\[\UM \bigwedge_{1\leq i,j\leq n;\; i \neq j} (\ind_i \limplies \neg \ind_j)\]

\item {Nodes have a unique depth:}
\[\UM \bigwedge_{0\leq i,j\leq b-1;\; i \neq j} (\rank_i \limplies \neg \rank_j)\]

\item {The depths partition the same set as the indexes:}
\[\UM (\beta \leftrightarrow \bigvee_{0\leq j\leq b-1} \rank_j)\]

\item {In every column $i$, at least one tile type was previously placed:}
\[\UM(\beta \rightarrow  \bigwedge_{0\leq i\leq n+1} (\col_i(T_0) \oor \dots \oor \col_i(T_{s+1}))).\]
\item {In every column $i$, at most one tile type was last placed:}
\[\UM(\col_i(T_u) \limplies \neg\col_i(T_v))\qquad(0 \leq i \leq n+1\text{ and }0 \leq u \neq v \leq s+1).\]
\item {Tile $T_0$ is already placed in columns $0$ and $n+1$:}
\[\UM (\beta \rightarrow (\col_0(T_0) \aand \col_{n+1}(T_0))).\]
\enumeratext
{
With these preliminaries behind us, we can now describe the structure of the game tree.
}
\item {No change in game position when depth is less than $b-1$:}
\[
\UM ( \beta\wedge\neg\rank_{b-1}\limplies (\eloise\limplies \Box\eloise) \aand (\neg \eloise \limplies \Box \neg \eloise))
\]
 and for $1 \leq i \leq n\text{ and }0 \leq u \leq s+1$
\[
 \UM ( \beta\wedge\neg\rank_{b-1}\limplies  (\col_i(T_u) \limplies \Box \col_i(T_u) \aand (\neg\col_i(T_u) \limplies \Box\neg \col_i(T_u))).
 \]

\item {In columns where no tile is placed, nothing changes when a move is made:}
  for all $1\leq i, j \leq n$ with $j\neq i$ and all $0\leq u\leq s+1$
\[\UM(\pos(i, b-1) \limplies (\col_j(T_u) \limplies \Box \col_j(T_u)) \aand (\neg\col_j(T_u) \limplies \Box \neg \col_j(T_u)))\]

\item {Players alternate:}
\[\UM(\rank_{b-1}\limplies (\eloise  \limplies \Box \neg \eloise) \aand (\neg \eloise \limplies \Box \eloise)).\]
\enumeratext{
Next, both players make legal moves; that is, they only place tiles that correctly match adjacent tiles. It will be helpful to define the following ternary relation of `compatibility' between tile types: 
\[C(T',T,T'') \iff \rright(T') = \lleft(T)\text{ and }\down(T) = \up(T'').\]
 That is, $C(T', T, T'')$ holds if and only if the tile $T$ can be placed to the right of tile $T'$ and above tile $T''$. With the aid of this relation we can formulate the first constraint on tile placement as follows.}

\item {Adjacencies:} for all $1\leq i\leq n$ and all tile types $T'$ and $T''$,
\[\UM(\pos(i, b-1) \aand \col_{i-1}(T') \aand \col_i(T'') \limplies \Box \Or \{\col_i(T) \mid C(T',T,T'')\}).\]
(Here, by convention, $\Or \emptyset = \bot$.)
\enumeratext{The  previous constraint only ensures matching to the left and downwards. We also need to ensure that tiles placed in column $n$ match the white wall tile to their right.}
\item {Match white on right in column $n$:}
\[\UM(\beta \rightarrow \Or\{ \col_n(T) \mid \rright(T)=\white\}).
\]
\item {All possible Abelard moves are included:} for all $1 \leq i < n$,
\begin{align*}
&\hspace{-17.36pt}\UM(\neg\eloise \aand \rank_{b-1} \aand \col_i(T'') \aand \col_{i-1}(T') \limplies\And \{\lozenge\col_i(T) \mid C(T', T, T'')\}),
\end{align*} 
(Here, by convention, $\And \emptyset = \top$).

\enumeratext{This completes the description of the game tree. The remaining conditions ensure Eloise has a winning strategy.}

\item {The initial position is a winning position for Eloise:}
\[\win.\]
\item {Recursive conditions:}
\[\UM(\win \limplies \begin{array}[t]{l}
(\col_1(T_{s+1}) \oor 
(\neg \eloise\aand \lozenge\top\wedge\Box \win)
 \oor (\eloise \aand \lozenge \win))).
 \end{array}\]
 \enumeratext{
To bound the length of the game, first define $N =2 n (s+2)^n $. If the game goes on for $N$ plays, then the position has repeated,\footnote{Including the initial position, $2 n (s+2)^n +1$ positions have occurred.}  which does not help Eloise: if she can win, she can do so in fewer than $N$ moves. Let $L$ be the number of binary digits necessary to write $N$.  We introduce new propositions $q_1, \ldots, q_L$ to denote the sequence of bits from the least to the most significant  digit of a $\counter$.}

\item {Counter is initially $0$:}
\[\neg q_L \aand \dots \aand \neg q_1.\]
\item {Counter does not change when $1\leq j\leq b-1$:}
\[
\UM(\beta\wedge \neg\rank_{b-1}\limplies \bigwedge_{k=1}^L((q_k\limplies\Box q_k)\wedge(\neg q_k\limplies\Box\neg q_k)).\]

\item {Incrementation of   counter when least-significant digit is $0$:}
\[
\UM(\rank_{b-1} \aand \neg q_1 \limplies  \Box q_1 \aand \And_{k =2}^L ((q_k \limplies \Box q_k) \aand (\neg q_k \limplies \Box\neg q_k))) .
\]
\item {Incrementation of counter when least-significant digit is $1$:}
for  $1\leq k< L$,
\begin{align*} 
\UM( &\rank_{b-1}\aand \neg q_{k+1} \aand \And_{l = 1}^k q_l \limplies
\\
&\Box ( q_{k+1} \aand \And_{l = 1}^k \neg q_l) \aand\And_{l=k+2}^L ((q_l\limplies \Box q_l )\aand (\neg q_l\limplies\Box \neg q_l) )).
\end{align*}
\item\label{form:last} {If the counter reaches $N$ then the game has gone on too long and Abelard wins:}
\[\UM( \pos \aand (\counter = N)  \limplies \neg\win).\]
\end{enumerate}

This concludes the definition of $\phi_\c I$.
We are now in a position to prove Theorem \ref{thm:main}.

\begin{proof}
We argue that whenever $(V, <)$ satisfies  conditions of the theorem (i.e. there is $\prime, F$ satisfying \eqref{cond}), then the derivation of $\phi_{\c I}$ from $\c I$ is a correct, polynomial-time reduction of the two-player corridor-tiling problem to satisfiability of $(\F, \P)$-formulas over $(V, <)$. Since the two-player corridor-tiling problem is EXPTIME-hard \cite{CHLEBUS1986374}, it follows that satisfiability of $(\F, \P)$-formulas over $(V, <)$ is too.

First note that the length of $\phi_{\c I}$ is polynomial in $s$ and $n$, at which point it is clear that $\phi_{\c I}$ can be calculated from $\c I$ in time polynomial in the size of (the encoding) of $\c I$.
We now argue that the reduction is correct.

\emph{Formula is satisfiable implies Eloise has a winning strategy:} Suppose $v$ is a valuation and $r \in V$ with $(V, <), v, r \models \phi_\c I$. Define a tree rooted at $r$ by setting the children of a node $x$ satisfying $\beta(i, j)$ to be \[\{y \models \beta(i, j)^+ \mid \exists z : z > x \wedge z > y \wedge \forall w(w < z \implies w \not\models f)\}\text{.}\] By the construction of $\phi_\c I$, this determines a labelled tree $\c B$ encoding a portion the game tree $\c T(\c I)$ sufficient for defining a strategy.  If this $\c B$ is not technically a tree because there is some unwanted sharing of descendants, then unwind it into the tree of paths in this directed graph that depart from $r$. Note also that there is no particular reason for $\c B$ to be binary. The formula $\phi_{\c I}$ further ensures that the initial position  is a position from which Eloise can force a win. Hence Eloise has a winning strategy.

\emph{Eloise has a winning strategy implies formula is satisfiable:} Assume the condition in the statement of Theorem~\ref{thm:main}.   Let $\c T(\c I)$ be a  labelled game tree for the instance $((T_0, \ldots, T_{s+1}), (I_1, \ldots, I_n))$.   As we saw, we can encode $\c T(\c I)$ as a labelled binary tree $\c B$ in which each node has a  position and depth, and by the hypothesis of the theorem, there is an embedding $'$ of $\c B$ as  $\c B'\subseteq V$  together with a subset $F$ of $V$ satisfying condition~\eqref{cond}.   

 Define a propositional valuation $v:\mathrm{Prop} \to \wp(V)$ by:  
\begin{itemize}
\item $v(f)=F$, \; 
\item$ v(\ind_i)=\set{b'\in \c B'\mid b\mbox{ has column-index $i$}}$,\;
\item $v(\rank_j)=\set{b'\in \c B'\mid b\mbox{ has depth $j$}}$,\; 
\item  $v(\eloise)$ consists of points in $\c B'$ where it is Eloise's turn,
  \item $v(\col_i(T))=\set{b'\in \c B'\mid T\mbox{ was last placed in column $i$ at $b$}}$.
  \end{itemize}  We may define $v(q_i)$ for $1\leq i\leq L$   so that the $q_i$s that hold at $b'$ encode, as a binary number,  the minimum of $2^L-1$ and 
  the number of edges on a path from $b$  to the root  $r$ of $\c T(\c I)$.   Let $\win$ hold at all points in the range of $'$ where Eloise has a winning strategy and the counter is less than $N$, but nowhere else.

Let $b$ be any node of the binary tree $\c B$, say $b'\models\beta(i, j)$ (some $i, j$).  Then $b$ is adjacent to two child nodes $c_1, c_2$ and perhaps to its parent, but no other node, and $\beta(i,j)^+$ holds at $c_1'$ and at $c_2'$.
By \eqref{eq:loz} and \eqref{cond},  if
$\varphi$ holds at $c_1'$ or at $c_2'$ then $\lozenge\varphi$ holds at $b'$.  Conversely, if $\lozenge\varphi$ holds at $b'$ then by  \eqref{eq:loz}   and \eqref{cond}, we know that $\beta(i,j)^+\wedge \varphi$ must hold where a node adjacent to $b$ in $\c B$ embeds. However,  $\beta(i,j)^+$ does not hold where the parent of $b$ embeds, so  $\varphi$ must hold at $c_1'$ or at $c_2'$.

 By the assumption that Eloise has a winning strategy, $r'\models\win$; other conjuncts of $\phi_{\c I}$ are now easily verified.  Thus $(V, <), v, r'\models \phi_{\c I}$, as required. 
\end{proof}

\section{Minkowski spacetime frames}\label{sec:minkowski}

In this section we apply Theorem~\ref{thm:main} to our target frames, thereby proving EXPTIME lower bounds.

\begin{lemma}\label{lem:circ}
Let $ a, b, c, d$ be  points on a circle $C$, in that order, as you go round the circle (clockwise or anticlockwise).  Every disc containing $a$ and $c$ also contains $b$ or $d$.

\end{lemma}

\begin{proof}
For contradiction, suppose $D$ is a disc containing $a$ and $c$ but neither $b$ nor $d$. Since $b\not\in D$, but $a\in D$, the boundary of $D$ must cross $C$ between $a$ and $b$.  Similarly, the boundary of $D$ meets $C$ three more times, between $b$ and $c$, between $c$ and $d$, and between $d$ and $a$.  Since $a, c\in D$ and $b, d\not\in D$, the four points in $C$ on the boundary of $D$ must be distinct.  But two circles intersecting more than twice must be identical, contradicting $a, c\in D$ and  $b, d\not\in D$.
\end{proof}

\begin{theorem}\label{thm:spacetime}
Validity of temporal formulas over $(m+1)$-dimensional Min\-kow\-ski spacetime is EXPTIME-hard for $m\geq 2$ with `lightspeed-or-slower' or with `slower-than-lightspeed' accessibility, 
 either the reflexive or irreflexive versions.
\end{theorem}

\begin{proof}
Let $(V, <)$ be any of the frames in the Theorem; clearly the frame is transitive and confluent. 
  To apply Theorem~\ref{thm:main}, we need to find an embedding $'$ from  the complete binary tree $\c B$ of depth $\omega$ into $V$ together with a subset $F\subseteq V$, in such a way that condition~\eqref{cond}  is satisfied. To do this, we will embed $\c B$ in a  two-dimensional spacelike plane $P$ perpendicular to the time axis and select $F\subseteq P$. We could, for example, assume $P$ is the plane defined by  $x_3 = x_4 =\dots=x_m = t = 0$. Points in $P$ are determined by a pair $(x, y)$ of real coordinates with respect to some fixed choice of orthonormal basis for $P$ (equipped with the standard inner product). Intersections of past light cones with $P$ are  discs (closed or open depending on whether lightspeed is included in $<$).   In the following, $D$ will range over all discs of the appropriate type---closed if lightspeed is included, open otherwise.  

To apply Theorem~\ref{thm:main}, we need  to establish condition~\eqref{cond}, which is equivalent to the following, for all $p\neq q\in \c B$,
\begin{equation}\label{eq:D} (p, q)\in E(\c B)\vee (q, p)\in E(\c B) \iff \exists D (p', q'\in D\wedge (D\cap F=\emptyset))
\end{equation}
For $p', q'$ let $p'\sim q'\iff \exists D(p', q'\in D\wedge (D\cap F=\emptyset))$.
An embedding $'$ and forbidden set $F$ satisfying~\eqref{eq:D}  is shown in Figure~\ref{fig:B}.

The nodes of $\c B$ may be written as finite strings of bits, and the two children of string $s$ are  obtained by appending $0$ or $1$ to the end of $s$.  For the embedding, map string $s=(b_0b_1\ldots b_{k-1})$ to $\sum_{i<k,\; b_i=0}2^{-i}(1,0)+ \sum_{i<k,\;b_i=1} 2^{-i}(0,1)$, so the root (the empty string) maps to $(0, 0)$, its two children $0$ and $1$ map to $(1, 0)$ and $(0, 1)$ respectively, its grandchildren $00$, $01$, $10$, and $11$ map to $(\frac32,0)$, $(1, \frac12)$, $(\frac12, 1)$, and $(0, \frac32)$ respectively, and so on; see Figure~\ref{fig:B}.

  The lines $y=x-2$ and $y=x+2$ are included in $F$.  Also, if $\mu$ is the midpoint in $P$ of (the images of) a pair of  siblings, every point  at or  above/right of $\mu$ on the line $y=x+c$ passing through $\mu$, is contained in $F$.  Nothing else is included in $F$.   These forbidden lines are shown in red in Figure~\ref{fig:B}.
\begin{figure}
\begin{center}
\begin{tikzpicture}[scale=1.3]
\draw (0,0) -- (1,0);
\draw (0,0)-- (0,1);
\filldraw[black] (0,0) circle (1.5pt) node[anchor=north] {Root};
\filldraw[black] (1,0) circle (1pt) node[anchor=north] {0};
\filldraw[black] (0,1) circle (1pt) node[anchor=east] {1};
\filldraw[black] (0,1.5) circle (.7pt) node[anchor=east] {{\tiny {11}}};
\filldraw[black] (0.5,1) circle (.7pt) node[anchor=south] {{\tiny {10}}};

\draw (0,1.5)--(0, 1.75);
\draw(0,1.5)--(.25,1.5);

\draw[red](.25,1.25)--(1,2);
\draw[red](1.25,.25)--(2,1);

\draw(0,1)--(0, 1.5);
\draw(0, 1)--(.5,1);
\draw(1,0)--(1.5,0);
\draw(1, 0)--(1, .5);

\draw[red](.5,.5)--(2,2);
\draw[red](1, -1)--(2,0);
\draw[red](-1,1)--(0,2);
\filldraw[red] (.5,.5) circle (1pt) node[anchor=east] {};

\filldraw[red] (.25,1.25) circle (.7pt) node[anchor=east] {};
\filldraw[red] (1.25,.25) circle (.7pt) node[anchor=east] {};

\draw[dashed](-.25,0.5) circle (.667);

\filldraw[black] (-.25,.5) circle (.5pt) ;


\end{tikzpicture}
\hspace{.9in}
\begin{tikzpicture}[scale=1.3]
\draw (0,0) -- (1,0);
\draw (0,0)-- (0,1);
\filldraw[black] (0,0) circle (1.5pt) node[anchor=north] {Root};
\filldraw[black] (1,0) circle (1pt) node[anchor=north] {0};
\filldraw[black] (0,1) circle (1pt) node[anchor=east] {1};
\filldraw[black] (0,1.5) circle (.7pt) node[anchor=east] {{\tiny {11}}};
\filldraw[black] (0.5,1) circle (.7pt) node[anchor=south] {{\tiny {10}}};

\draw (0,1.5)--(0, 1.75);
\draw(0,1.5)--(.25,1.5);

\draw(0.5,1)--(0.75,1);

\draw(0.5,1)--(0.5,1.25);

\draw[red](.25,1.25)--(1,2);
\draw[red](1.25,.25)--(2,1);

\draw(0,1)--(0, 1.5);
\draw(0, 1)--(.5,1);
\draw(1,0)--(1.5,0);
\draw(1, 0)--(1, .5);
\draw[red](.5,.5)--(2,2);
\draw[red](1, -1)--(2,0);
\draw[red](-1,1)--(0,2);
\filldraw[red] (.5,.5) circle (1pt) node[anchor=east] {};

\filldraw[red] (.25,1.25) circle (.7pt) node[anchor=east] {};
\filldraw[red] (1.25,.25) circle (.7pt) node[anchor=east] {};

\filldraw[red] (-.75,1.25) circle (.74pt) node[anchor=east] {};


\draw(-.219,.719) circle(.75);


\end{tikzpicture}
\end{center}
\caption{\label{fig:B}Embedding of  complete binary tree $\c B$, and forbidden set $F$ (shown in red) in a plane, to scale.  Left:  parent $\sim$ child; a  disc containing root and child, disjoint from $F$.  Right:   grandparent $\not\sim$ grandchild; a circle through $(0, 0),    { \color{red}   (\frac12,\frac12)}$, and ${\color{red} (-\frac34,\frac54)}$; nodes $11, 10$, and descendants are outside the circle, between the two forbidden nodes.}
\end{figure}
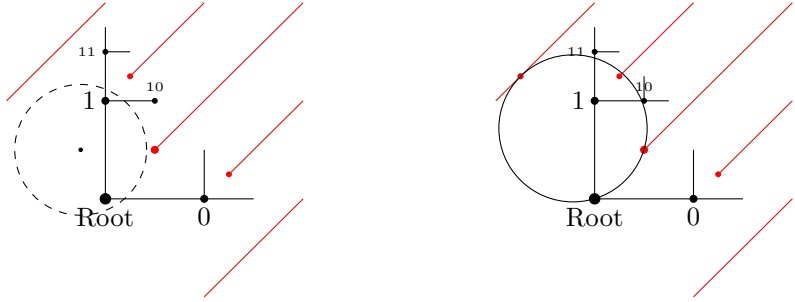

For the left-to-right implication of \eqref{eq:D}, first suppose $p$ is the root of $\c B$. Consider the open disc $O_{\frac23}(-\frac14,\frac12)$ shown on the left of Figure~\ref{fig:B} (to scale), and first observe that the root and its child $1$  map into the disc, since  $\frac{\sqrt 5}4 <\frac 23$. Then observe that all forbidden points are outside the closure of the disc, since the nearest point on the line $y = x+2$ is $(-\frac 7 8, \frac 9 8)$ at a distance $\frac{5\sqrt 2}8$, and the nearest points in $F$ on the lines $y = x+1$ and $y = x$ are the endpoints $(\frac 1 4, \frac 5 4)$ and $(\frac 1 2, \frac 1 2)$ at distances $\frac{\sqrt{13}}4$ and $\frac 3 4$, respectively. 
 Similarly,  the root and its other child, $0$, map into  a disc disjoint from $F$. Cases where $p$ is not the root are treated shortly.

For the right-to-left implication of \eqref{eq:D},  suppose $D$ avoids forbidden nodes, and $p'\neq q'\in D$.  By convexity of $D$ all points on the line segment from  $p'$ to $q'$ belong to $D$,  which is disjoint from $F$.    The midpoint of (the images of) two siblings is in $F$, so $p$ and $q$ cannot be siblings.    Moreover, if $p$ and $q$ are descendants of two distinct siblings  then there is a point on the line segment between $p'$ and $q'$  that is forbidden, so this case cannot happen either.  So either $p$ is a descendant of $q$, or $q$ is a descendant of $p$.

We claim that $p$ and $q$ cannot be more than one generation apart.  Again, we first prove this claim when $p$ is the root, which embeds to $(0, 0)$; thus $q$ is a descendent of the root (i.e.\ any node of the tree).    The circle $C_{\frac{\sqrt{49+529}}{32}}(-\frac7{32}, \frac{23}{32})$ meets $y=x+2$ tangentially at $(-\frac34,\frac54)$ and passes through the node at $(0, 0)$ and the forbidden point $(\frac12, \frac12)$. This is pictured in the right-hand diagram of Figure~\ref{fig:B}.  The forbidden point at $(\frac14, \frac54)$  happens to lie inside the circle; we make no use it here, but note this point only makes it harder to find discs avoiding forbidden nodes.  Here we use the forbidden points   $(-\frac34,\frac54)$ and $(\frac12, \frac12)$.
 All descendants of $11$, at $(0, \frac32)$, or of $10$, at $( \frac12, 1)$,  lie, from the point of view of $(0,0)$, between these two forbidden nodes and are outside the circle, since  $\frac{\sqrt{49+625}}{32}, \frac{\sqrt{529+81}}{32} > \frac{\sqrt{529+48}}{32}$. (Reading glasses may be needed to see that $10$ is outside the circle in Figure~\ref{fig:B}.)
     Thus every disc containing $(0, 0)$ and avoiding forbidden nodes, excludes $11$, $10$, and their descendants, else by convexity the disk would also contain a point $c$ on $C_{\frac{\sqrt{49+529}}{32}}(-\frac7{32}, \frac{23}{32})$, such that $(0,0)$,  $(-\frac34,\frac54)$, $c$, $(\frac12, \frac12)$ is the clockwise order, contradicting Lemma~\ref{lem:circ}. So $q$ is none of $11$, $10$, or their descendants.    
    By a symmetric argument, we see that $q$ cannot be $00$, $01$, or their descendants.    
It follows that $q$ must be a child, either $0$ or $1$, of the root, proving the right-to-left implication in \eqref{eq:D}, for the case where $p$ is the root.

To generalise to arbitrary $p$,  we now define a $P$-similarity.
 Recall that   $P\subseteq V=\reals^{m+1}$ is a two-dimensional spatial plane, the image of  ${^\prime}:\c B\to V$  is included in $P$, and  $F\subseteq V$.  Let $\c B'=\set{p' \mid p\in\c B}$ and $E'=\set{(p',q')\mid (p, q)\in E(\c B)}$.   A \emph{$P$-similarity} $\sigma:P\to P$ is a similarity of $P$ with respect to the Euclidean metric over $P\subseteq\reals^{m+1}$,   such that
\begin{itemize}
\item $p'\in  \c B'\implies \sigma(p')\in\c B'$ (`nodes map to nodes'),
\item $(p', q')\in E'\implies (\sigma(p'), \sigma(q'))\in E'$ (`edges map to edges'),
\item $\sigma(P \setminus F) \subseteq P \setminus F$ 
 (`permitted points map to permitted points').
\end{itemize}
for all $p, q\in\c B$. Observe that if $\sigma$ is a $P$-similarity and $p'\in D$ for some disc $D$ disjoint from $F$, then $\sigma[D]$ is a disc disjoint from $F$  containing $\sigma(p')$.
The map $(x, y)\mapsto (1+\frac x2, \frac y2)$ is a $P$-similarity mapping $(0, 0)$ to $(1, 0)\; (=0')$,
and likewise for all $p\in\c B$ there is a $P$-similarity mapping $(0, 0)$ to $p'$.  Hence the  left-to-right implication of \eqref{eq:D} holds for all $p\in\c B$.   The proof of the right-to-left implication of \eqref{eq:D} goes through when $p$ is not the root,  since there is a $P$-similarity mapping $(0, 0)$ to $p'$ carrying the forbidden points $(-\frac34,\frac54)$ and $(\frac12, \frac12)$ to forbidden points.
\end{proof}

\section{Open problems}\label{sec:problems}
\begin{enumerate}[(1)]
\item Our main result is the EXPTIME-hardness of the temporal validity problem over the Minkowski spacetimes $(\reals^n, \leq), (\reals^n, <), (\reals^n, \prec), (\reals^n, \preceq)$,  for $n\geq3$, but we have no upper bound on the complexity of these frames; we do not know if the logic of any of these frames is  decidable, nor even if they are recursively enumerable.   For $n=2$ (only one spatial dimension) we know that all four validity problems are  PSPACE-complete \cite{HR18,HM18}.   Find more precise bounds on the complexity of the validities of $(\reals^n, <)$, for $n\geq 3$.
\item There are purely modal formulas satisfiable in $(\reals^3, <)$ but not in $(\reals^2, <)$ \cite{Gol80}, and there are temporal formulas satisfiable in $(\reals^3,\leq)$ but not in $(\reals^2,\leq)$ \cite{HR18}.  Are there temporal formulas that distinguish $(\reals^3, <)$ from $(\reals^4, <)$?
\item Find sound and complete axioms for the temporal validities of any of these Minkowski frames. (These are open even with one spatial dimension.)
\end{enumerate}


\bibliographystyle{aiml22}
\bibliography{brettbib}

\end{document}